\documentclass{amsart}

\usepackage{url}
\usepackage[cmtip,all]{xy}

\newtheorem{theorem}{Theorem}[section]
\newtheorem{lemma}[theorem]{Lemma}
\newtheorem{proposition}[theorem]{Proposition}
\newtheorem*{theononum}{Theorem}

\theoremstyle{definition}

\numberwithin{equation}{section}

\newcommand{\de}{\colon}
\newcommand{\st}{\mid}
\newcommand{\C}{\mathbb{C}}
\newcommand{\Q}{\mathbb{Q}} 
\newcommand{\Ar}{\mathcal{A}}

\DeclareMathOperator{\rk}{rk}
\DeclareMathOperator{\codim}{codim}
\DeclareMathOperator{\sgn}{sgn}
\renewcommand{\phi}{\varphi}
\newcommand{\Lie}{\mathbb{L}}
\DeclareMathOperator{\id}{id}

\begin{document}
\title[Homotopy Lie algebra of geometric arrangements]{Homotopy Lie algebra of the complements of subspace arrangements with 
geometric lattices}

\author{Gery Debongnie}
\address{UCL, Departement de mathematique \\ Chemin du Cyclotron, 2 \\ 
B-1348 Louvain-la-neuve \\ Belgium}
\email{debongnie@math.ucl.ac.be}
\thanks{The author is an ``Aspirant'' of the ``Fonds National pour la 
Recherche Scientifique'' (FNRS), Belgium.}


\begin{abstract}
Let  $\Ar$ be a geometric  arrangement such that  $\codim(x) \geq 2$ for
every $x \in\Ar$.  We prove that, if the complement space $M(\Ar)$ is  rationally hyperbolic, then there exists an injective
map $\Lie(u,v) \to \pi_\star(\Omega M(\Ar)) \otimes \Q$.
\end{abstract}

\maketitle

\section{Introduction}
Let $\Ar = \{x_1, \dotsc, x_n\}$ be a subspace arrangement in $\C^l$ 
with a geometric lattice and such that for every $x \in \Ar$, $\codim(x)
\geq 2$.  In~\cite{yu05}, S. Yuzvinsky and E. Feichtner described a
simple rational model for the topological space $M(\Ar) = \C^l \setminus \cup_{x \in \Ar} x$.  In~\cite{de07}, we give a complete description of subspace arrangements with
a geometric lattice and with the codimension condition such that the
space $M(\Ar)$ is rationally elliptic.  Briefly, we prove in~\cite{de07} that $M(\Ar)$ is rationally elliptic if and only if $M(\Ar)$ has the homotopy type of a product of odd-dimensional spheres.

Now, since we know the elliptic case quite well, the next step is to try to understand what happens when $M(\Ar)$ is rationally hyperbolic. It is a much more general situation, so, we can expect it to be much more complicated.  The main result of this paper is the following theorem~:
\begin{theononum}
Let $\Ar$ be a geometric arrangement such that for every $x \in \Ar$, we
have $\codim(x) \geq 2$. If $M(\Ar)$ is rationally hyperbolic, then there exists
an injective map $\Lie(u,v) \to \pi_\star(\Omega M(\Ar)) \otimes \Q$.
\end{theononum}

The rational model of the space $M(\Ar)$ given by Yuzvinsky is described in 
section~\ref{sec:rmsa}. In section~\ref{sec:generalsit}, the general
situation is set up~: a map $\phi \de \Lambda (e_1, \dotsc, e_n) \to H^\star(M(\Ar), \Q); e_i \mapsto [\{x_i\}]$ is defined and studied. This map and its kernel will play an important role in the proof.  Finally, the last two sections contain the proof of the theorem. 

\section{Rational model of subspace arrangements} \label{sec:rmsa}
Let  $\Ar$ be a  central arrangement  of subspace  in $\C^l$.  It is known that, with the appropriate choice of the operations $\vee$ and $\wedge$, the set $L(\Ar)$ of non empty intersections of elements of $\Ar$ is a lattice with a rank function. Yuzvinsky
defined the relative atomic differential graded algebra $(D_\Ar,d)$
associated with  an arrangement as follows (see~\cite{yu05})~:  choose a
linear  order on $\Ar$.  The graded vector space $D_\Ar$ has a basis given by  all
subsets $\sigma \subseteq \Ar$.  For $\sigma = \{x_1, \dotsc, x_n\}$, we
define the differential by the formula \[
d\sigma = \sum_{j:\vee(\sigma \setminus \{x_j\}) = \vee \sigma} (-1)^j
(\sigma\setminus\{x_j\})
\]
where the indexing of the elements  in $\sigma$ follows the linear order
imposed on $\Ar$. With $\deg(\sigma) = 2 \codim \vee \sigma - |\sigma|$,
$(D_\Ar,d)$  is a  cochain complex.   Finally,  we need a multiplication  on
$(D_\Ar,d)$.  For $\sigma, \tau \subseteq \Ar$, 
\[
\sigma \cdot \tau = \left\{\begin{aligned} (-1)^{\sgn\epsilon(\sigma, 
\tau)} \sigma \cup \tau &\text{ if } \codim \vee \sigma + \codim \vee 
\tau = \codim \vee(\sigma \cup \tau) \\ 0 & \text{ otherwise} 
\end{aligned}\right.
\]
where  $\epsilon(\sigma, \tau)$ is the  permutation  that, applied  to $
\sigma \cup  \tau$ with the  induced linear  order, places  elements  of
$\tau$  after elements  of $\sigma$, both  in the induced  linear order.

A subset  $\sigma \subseteq \Ar$ is  said to be independant if $\rk(\vee
\sigma)  =  |\sigma|$.  When $\Ar$  is a a subspace  arrangement with  a
geometric lattice, then $H^\star(M(\Ar))$ is
generated by the classes $[\sigma]$, with $\sigma$ independant 
(see~\cite{yu05}). 

\section{General situation} \label{sec:generalsit}
Let  $\Ar = \{x_1,  \dotsc,  x_n\}$ be  a subspace  arrangement  with  a
geometric lattice such that every $x \in \Ar$ has $\codim(x) \geq 2$. We will suppose that no element $x_i$ is contained in another one, because otherwise, we can omit it when we consider $M(\Ar)$. We
 consider the morphism of graded algebras \[
\phi \de \Lambda (e_1, \dotsc, e_n) \to H^\star(M(\Ar), \Q); e_i \mapsto
 [\{x_i\}]. 
\]
As we will see,  in some sense,  the kernel of this map measure the non-ellipticity of the space $M(\Ar)$. The following proposition shows a clear connection between $\ker \phi$ and ellipticity.
\begin{proposition} \label{prop:injrat}
If the map $\phi$ is injective, then the space $M(\Ar)$ is rationally 
elliptic.
\end{proposition}
\begin{proof}
If this map is injective,  then,  for each sequence $1 \leq i_1 < i_2 < \dotso < i_s \leq n$, we have $\{x_{i_1} \} \cdot \{x_{i_2}\} \cdot \dotsc \cdot
\{x_{i_s}\}  \neq  0$ because  their  product  is non  zero  in  cohomology.
Therefore, for an appropriate choice of sign, we have the following equality  $\prod_{j=1}^s \{x_{i_j}\} = \pm
\{x_{i_1},  x_{i_2},  \dotsc,  x_{i_s}\}$ and $[\{x_{i_1},  \dotsc,  x_{i_s}\}] \neq  0$ (in
cohomology).  This implies that $\phi$  is surjective because,  for each
independant    set    $\{x_{i_1},   \dotsc,    x_{i_s}\}$    (which    generates
$H^\star(M(\Ar))$),  we have $[\{x_{i_1}, \dotsc, x_{i_s}\}] = \pm \prod_{i=j}^s
[\{x_{i_j}\}]$, which is in the image of $\phi$.  It means that $\phi$ is an
isomorphism.   Therefore,  $M(\Ar)$ has the rational  homotopy type of a
product  of  odd  dimensional spheres  and  theorem~5.1  of~\cite{de07}
implies that $M(\Ar)$ is rationally elliptic.  
\end{proof}

Now,  assume that the map $\phi$ is not injective.  In that case,
we can  define the natural  number $r = \max\{ s  \st \ker \phi  \subset
\Lambda^{\geq s} e_i \}$.  It is clear that $2 \leq r \leq n$. The bigger $r$ is, the smaller $\ker \phi$ is. Also, we understand quite well $\phi(\Lambda^{\leq r} e_i) \subset D_\Ar$~:
\begin{lemma} \label{lem:diff}
If $\sigma \in D_\Ar$ with $|\sigma| \leq r$, then 
$d\sigma = 0$ and $\rk \vee \sigma = |\sigma|$.
\end{lemma}
\begin{proof}
We use induction on $s$ to prove that for $1 \leq s < r$ and for each sequence $1 \leq i_1 < i_2 < \dotso < i_s \leq n$, 
\begin{enumerate}
\item $d\{x_{i_1}, \dotsc, x_{i_s}\} = 0$,
\item $\phi(e_{i_1} \dotso e_{i_s}) = 
[\{x_{i_1}, \dotsc, x_{i_s} \}] \neq 0$.
\end{enumerate}
It is true for $s = 1$. Now suppose that it is true for $s-1$.  If 
$d\{x_{i_1}, \dotsc, x_{i_s}\} \neq 0$, then $d\{x_{i_1}, \dotsc, 
x_{i_s}\}$ is a non zero linear combination $ \sum \rho_j \{x_{j_1}, 
\dotsc, x_{j_{s-1}}\}$ and \[
0 = \left[ \sum \rho_j \{x_{j_1}, \dotsc, x_{j_{s-1}} \}\right] = \phi
\left( \sum \rho_j \,e_{j_1} \dotso e_{j_{s-1}}\right)
\]
which is impossible because $\phi$ restricted to $\Lambda^{<r}(e_1, \dotsc, e_n)$ is injective. This shows that 
$d\{x_{i_1}, \dotsc, x_{i_s}\} = 0$. 

The map $\phi$ is extended in a multiplicative way, 
therefore, by the induction hypothesis, we have~: \[
\phi(e_{i_1} \dotso e_{i_s}) = \phi(e_{i_1}) \phi(e_{i_2}\dotso e_{i_s})
 = [\{x_{i_1}\}][\{x_{i_2} \dotso x_{i_s}\}].
\]
But $s < r$, so $\phi(e_{i_1} \dotso e_{i_s}) \neq 0$ and we have $ \phi(e_{i_1} \dotso e_{i_s}) = [\{x_{i_1}, \dotsc, x_{i_s} \}]$. This proves the assertion (2). This proof by induction showed that $d\sigma = 0$ if $|\sigma| < r$. But the exact same reasoning can be done for $|\sigma| = r$. So, $d \sigma = 0$ if $|\sigma| \leq r$.

In order to prove that $\rk \vee \sigma = |\sigma|$, let's prove by induction that if $1 \leq s \leq r$, then for each sequence $1 \leq i_1 < i_2 < \dotsc < i_s \leq n$, $\rk \vee \{
x_{i_1}, \dotsc, x_{i_s} \} = s$.  It is obviously true for $s = 1$.  
Assume that it is true until $s-1 < r$.  By the induction hypothesis, $\vee 
\{x_{i_1}, \dotsc, x_{i_{s-2}}\} < \vee\{x_{i_1}, \dotsc, x_{i_{s-2}}, 
x_{i_s}\}$ is a maximal chain (if $s=2$, then $\vee\{x_{i_1}, \dotsc, 
x_{i_{s-2}}\} = \vee \emptyset = \C^l$).  But the lattice $L(\Ar)$ is 
geometric. So, \[
\vee \{x_{i_1}, \dotsc, x_{i_{s-2}}\} \vee x_{i_{s-1}} \leq \vee\{
x_{i_1}, \dotsc, x_{i_{s-2}}, x_{i_s}\} \vee x_{i_{s-1}}
\]
is also a maximal chain. The first part of this lemma shows that $d\{x_{i_1},
\dotsc, x_{i_{s}}\} = 0$,  which implies that \[
\vee \{x_{i_1}, \dotsc, x_{i_{s-1}}\} \neq \vee\{x_{i_1}, \dotsc, 
x_{i_{s-2}}, x_{i_{s-1}}, x_{i_s}\}.
\]
Hence, $\rk \vee\{x_{i_1}, \dotsc, x_{i_s}\} = \rk \vee\{
x_{i_1}, \dotsc, x_{i_{s-1}}\} + 1 = s$.
\end{proof}

To make the next sections easier to read, we will use the following notations. For a commutative 
differential graded algebra $(A,d)$, let's denote by $L_{(A,d)}$ the homotopy 
Lie algebra associated to its Sullivan minimal model. And for every $1 \leq 
i_1 < i_2 < \dotso < i_{r+1} \leq n$, let's denote by $[e_{i_1}, \dotsc, 
e_{i_{r+1}}]$ the element \[
\sum_{j=1}^{r+1} (-1)^j e_{i_1} \dotso \widehat{e_{i_j}} \dotso e_{i_{r+1}}.
\]

\section{Main result} \label{sec:main}

We will study the situation described in section~\ref{sec:generalsit} 
with $\ker \phi \neq 0$ (if $\ker \phi = 0$, proposition~\ref{prop:injrat} shows that we are in the elliptic case, which is studied in~\cite{de07}).  There are two slightly different cases that 
can arise : either $\ker \phi$ contains a $r$-uple $e_{i_1} \dotso 
e_{i_r}$ or $\ker \phi$ does not contain such a $r$-uple. The next two propositions shows the existence of an injective map $\Lie(u,v) \to \pi_\star \Omega M(\Ar) \otimes \Q$ in these two cases. Then, the main theorem~\ref{theo:main} is proved.

\begin{proposition} \label{prop:a}
If $\ker \phi$ contains a $r$-uple $e_{i_1}\dotso e_{i_r}$ with $1 \leq i_1 < \dotso < i_r \leq n$, then there 
exists an injective map \[
\Lie(u,v) \to \pi_\star \Omega M(\Ar) \otimes \Q.
\]
\end{proposition}
\begin{proof}
We define $(A_4,0) = \left(\frac{\Lambda(e_{i_1}, \dotsc, e_{i_r})}{e_{i_1} 
\dotso e_{i_r}}, 0\right)$ and we construct the map $\psi \de (D_\Ar, d) \to 
(A_4, 0)$  in the following way~: if $\{k_1, \dotsc, k_t\} \subseteq 
\{i_1, \dotsc, i_r\}$ and $k_1 < \dotso < k_t$, then $\psi(\{k_1, 
\dotsc, k_t\}) = [e_{k_1}\dotso e_{k_t}]$. Otherwise, $\psi(\{k_1, 
\dotsc, k_t\}) = 0$. Since $\ker \phi \cap \Lambda^{<r} (e_1, \dotsc, 
e_n) = 0$, a simple check shows that $\psi$ is multiplicative. 
Lemma~\ref{lem:diff} shows that $\psi(d\sigma) = \psi(0) = 0 = 
d\psi(\sigma)$.  Hence, $\psi$ is a morphism of differential graded 
algebras.

Since $e_{i_1}\dotso e_{i_r} \in \ker \phi$, we can define another map 
$ \rho \de (A_4,d) \to H^\star((D_\Ar, d), \Q)$ by letting $\rho(
[e_{i_s}]) = [\{x_{i_s}\}]$. This is a morphism of graded algebras. Now, we 
have the following maps~: \[
A_4 \stackrel{\rho}{\to} H^\star((D_\Ar, d), \Q) \stackrel{H^\star\psi}
{\longrightarrow} A_4.
\]
Those maps verify the following property~: $(H^\star\psi) \circ \rho = 
\id$, which means that $H^\star \psi$ is a retraction of $\rho$. Since 
$M(\Ar)$ is a formal space (proved in~\cite{yu05}), the 
lemma~\ref{lem:magic} implies then the existence of an injective map $h 
\de L_{(A_4,0)} \to \pi_\star \Omega M(\Ar) \otimes \Q$.  By 
lemma~\ref{lem:a}, there is an injective map $\Lie(u,v) \to 
L_{(A_4,0)}$.  The composition of these two maps gives us the needed 
application.
\end{proof}

\begin{proposition} \label{prop:b}
If $\ker\phi$ does not contain a $r$-uple $e_{i_1}\dotso e_{i_r}$, then 
there exists an injective map \[
\Lie(u,v) \to \pi_\star \Omega M(\Ar) \otimes \Q.
\]
\end{proposition}
\begin{proof}
Since $\ker \phi \cap \Lambda^r(e_1, \dotsc, e_n) \neq \emptyset$, there
exists a non zero linear combination $\sum \lambda_{i_1\dotso i_r} e_{i_1}\dotso e_{i_r}$ such 
that $\phi(\sum \lambda_{i_1 \dotso i_r} e_{i_1} \dotso e_{i_r}) = 0$. 
So, $[\sum \lambda_{i_1 \dotso i_r} \{x_{i_1}, \dotsc, x_{i_r}\}] = 0$ 
in $H^\star(D_\Ar, d)$ and there exists a $\sigma \in D_\Ar$ such that 
$d\sigma = \sum \lambda_{i_1 \dotso i_r} \{x_{i_1}, \dotsc, x_{i_r}\} 
\neq 0$. From this, we deduce that there exists $1 \leq i_1 < \dotsc < i_{r+1} \leq n$ 
such that $d\{x_{i_1}, \dotsc, x_{i_{r+1}}\} \neq 0$. 

Let $X = x_{i_1} \vee x_{i_2} \vee \dotso \vee x_{i_{r+1}}$ and $B = \{x
 \in \Ar \st x < X \} = \{x_{j_1}, \dotsc, x_{j_m} \}$.  Using
lemma~\ref{lem:diff} and the fact that $d\{x_{i_1}, \dotsc, x_{i_{r+1}}\} \neq 0$, we observe that $\rk X = r$.  Also, lemma~\ref{lem:diff} shows that for any subset $\sigma \subset B$ with $r$ 
elements, $\rk \vee \sigma = r = \rk X$, so $\vee \sigma = X$.  It  implies that any $r+1$ product 
$\prod_{i=1}^{r+1} \{x_{k_i}\} =0$ for $x_{k_i}$ in $B$.  It allows us to define the following map~: \[
\rho \de \frac{\Lambda(e_{j_1}, \dotsc, e_{j_m})}{\Lambda^{\geq r+1} (
e_{j_1}, \dotsc, e_{j_m})} \to H^\star(D_\Ar, d); e_j \mapsto [\{x_j\}].
\]
Let's prove that $\ker \rho \subset \Lambda^r(e_{j_1}, \dotsc, e_{j_m})$
is generated by the $[e_{i_1}, \dotsc, e_{i_{r+1}}]$ with $\{i_1, 
\dotsc, i_{r+1}\} \subseteq \{j_1, \dotsc, j_m\}$~:
\begin{itemize}
\item It is clear that $\rho[e_{i_1}, \dotsc, e_{i_{r+1}}] = d\{x_{i_1},
\dotsc, x_{i_{r+1}}\}$ (because $\rk X = r$, $\ker \phi$ does not contain any $r$-uple and, by 
lemma~\ref{lem:diff}, $\rk\vee\{x_{i_1}, \dotsc, \widehat{x}_{i_j}, \dotsc, 
x_{i_{r+1}}\} = r$).
\item If $\{i_1, \dotsc, i_r\} \subseteq \{j_1, \dotsc, j_m\}$ and $y 
\in \Ar \setminus B$, then $d\{x_{i_1}, \dotsc, x_{i_r}, y\}$ is a sum with no term equal to $\{x_{i_1}, \dotsc, x_{i_r}\}$.  
Therefore, if $u \in \ker \rho$, then $\rho u = d \sigma$ where $\sigma$
 is a linear combination of $\{x_{i_1}, \dotsc, x_{i_{r+1}}\}$ with $
\{i_1, \dotsc, i_{r+1}\} \subseteq \{j_1, \dotsc, j_m\}$. In other 
words, since $\ker \phi$ does not contain any $r$-uple, $u$ is a linear combination of $[e_{i_1}, \dotsc, e_{i_{r+1}}]$, 
as required.
\end{itemize}
Let $A_5 = \frac{\Lambda(e_{j_1}, \dotsc, e_{j_m})}{\Lambda^{\geq r+1} (
e_{j_1}, \dotsc, e_{j_m}) \oplus \ker \rho}$. The map $\rho$ induces an 
injective map $\bar{\rho}$, and we define a map $\psi$ in the opposite 
direction \[
A_5 \stackrel{\bar{\rho}}{\to} H^\star(D_\Ar, d) \stackrel{\psi}{\to}A_5
\]
by sending $\{x_i\}$ to $[e_i]$  if $i \in \{j_1, \dotsc, j_m\}$ and 
zero if $i \not \in \{j_1, \dotsc, j_m\}$.  These two maps are morphisms
 of graded algebras and verify the following property~: $\psi \circ 
\bar{\rho} = \id$.  Finally, the lemmas~\ref{lem:b} and~\ref{lem:magic} give us two 
injective maps $\Lie(u,v) \to L_{(A_5,0)} \to \pi_\star \Omega M(\Ar) 
\otimes \Q.$
\end{proof}

With the two previous propositions, the next theorem is almost completely proved. We just need to put everything in place.

\begin{theorem} \label{theo:main}
Let $\Ar$ be a geometric arrangement such that every $x \in \Ar$ has 
$\codim(x) \geq 2$. Then $M(\Ar)$ is rationally hyperbolic if and only if there is 
an injective map $\Lie(u,v) \to \pi_\star(\Omega M(\Ar)) \otimes \Q$.
\end{theorem}
\begin{proof}
Suppose that $M(\Ar)$ is rationally hyperbolic. As shown at the beginning of this 
section, the map $\phi\de \Lambda(e_1, \dotsc, e_n) \to H^\star(M(\Ar), 
\Q)$ can not be injective, otherwise $M(\Ar)$ would be elliptic. 
Therefore $\ker \phi \neq 0$ and the propositions~\ref{prop:a} 
and~\ref{prop:b} show that there exists an injective map $\Lie(u,v) \to 
\pi_\star(\Omega M(\Ar)) \otimes \Q$.

Now, assume that such a map exists. In that case, the dimension of $\pi_\star(\Omega M(\Ar))\otimes \Q$, as a graded rational vector space, is not finite. Hence, the same is true for $\pi_\star M(\Ar) \otimes \Q$ and, by the dichotomy theorem in
rational homotopy theory, $M(\Ar)$ is rationally hyperbolic.
\end{proof}

\section{Technical results} \label{sec:technical}
This section contains the technical lemmas concerning $A_4$ and $A_5$ used in section~\ref{sec:main}. The aim is to prove the lemmas~\ref{lem:a},~\ref{lem:b},~\ref{lem:magic}. With that in mind, we consider the following differential graded algebras~:
\begin{align*}
(A_1, 0) &= \bigg( \Lambda(e_{i_2}, \dotsc, e_{i_r}) \oplus  (\oplus_{s 
\geq 1} \Q u_s), 0  \bigg), |u_s| = \sum\nolimits_{i=1}^r |e_{i_r}| + 
(s-1)|e_{i_1}| - s,\\
(A_2, d) &= \left( \frac{\Lambda(e_{i_1}, \dotsc, e_{i_r}, t, a)}{t
e_{i_1}, \dotsc, te_{i_r}, t^2}, d\right) \text{ with } de_{i_j} = 0, dt
= e_{i_1}\dotso e_{i_r}, da = e_{i_1}, \\
(A_3, d) &= \left( \frac{\Lambda(e_{i_1}, \dotsc, e_{i_r}, t)}{te_{i_1},
\dotsc, te_{i_r}, t^2}, d\right) \text{ with } de_{i_j} = 0, dt = 
e_{i_1}\dotso e_{i_r}, \\
(A_4, 0) &= \left(\frac{\Lambda(e_{i_1}, \dotsc, e_{i_r})}{e_{i_1}
e_{i_2} \dotso e_{i_r}}, 0 \right), \\
(A_5,0) &= \left( \frac{\Lambda(e_{j_1}, \dotsc, e_{j_m})}{I}, 0 
\right).
\end{align*}
where $I$ is the ideal of $\Lambda(e_{j_1}, \dotsc, e_{j_m})$ generated by
 the elements $e_{i_1} \dotso e_{i_{r+1}}$ and $[e_{i_1}, \dotsc, 
e_{i_{r+1}}]$. In $(A_1,0)$, the products $u_se_{i_j} = 0$ and $u_su_{s'} = 0$ for all 
$s, s'$ and $j$. Remark that $(A_2,d)$ is equal to $(A_3\otimes \Lambda 
a, d)$ with $da = e_{i_1}$. 

In order to reach our goal, we will need to understand a few properties of these algebras. The proofs make heavy use of rational homotopy theory (especially Sullivan minimal models). The theory and notations are explained in~\cite{fe00}.

\begin{lemma} \label{lem:algrel}
There exists two quasi-isomorphisms $(A_1, 0) \stackrel{\simeq}{\to} 
(A_2, d)$ and $(A_4, 0) \stackrel{\simeq}{\to} (A_3,d)$.
\end{lemma}
\begin{proof}
It is clear that the inclusion $(A_4,0) \to (A_3, d)$ is a quasi-isomorphism, because, as a vector space, $A_3 = A_4 \oplus V$ where  $V$
admits $e_{i_1} \dotso e_{i_r}$ and $t$ as basis elements. Let's prove 
that there exists a quasi-isomorphism $\theta\de (A_1,0) \to (A_2, d)$. 
Consider the subalgebra $(B,d) = (\Lambda(e_{i_1}, \dotsc, e_{i_r}, a), 
d)$ of $(A_2, d)$.  Since $d(A_2) \subset B$, the differential in 
$A_2/B$ is zero.  Therefore, we have a short exact sequence of 
complexes~: \[
0 \to (B,d) \to (A_2, d) \to \left( A_2/B, 0\right) \to 0,
\]
and a long exact sequence in cohomology with $A_2/B = \oplus_{s \geq 0} 
\Q ta^s$.  By the connecting map, an element $ta^s$ of $H^\star(A_2/B, 0
)$ is sent on the cohomology class of $d(ta^s) = e_{i_1} \dotso e_{i_r} 
a^s$ in $B$. But $d\big(\frac{1}{s+1} e_{i_2} \dotso e_{i_r}a^{s+1}\big) = 
e_{i_1} \dotso e_{i_r}a^s$. Therefore, the connecting map is zero. It 
means that we have a short exact sequence of the cohomology algebras~: 
\[
0 \to H^\star(B,d) \to H^\star(A_2, d) \to H^\star(A_2 / B, 0) \to 0.
\]
The cohomology of $(B,d)$ is obviously $\Lambda(e_{i_2}, \dotsc, e_{i_r}
)$ and the cohomology of $(A_2/B, 0)$ is $A_2/B$.  Consider the map 
$\theta \de (A_1, 0) \to (A_2, 0)$ defined by $\theta(e_{i_j}) = 
e_{i_j}$, $j = 2, \dotsc, n$ and $\theta(u_s) = \frac{a^s}{s}e_{i_2} 
\dotso e_{i_r} - a^{s-1}t$. It is a morphism of differential graded algebras. This gives us the
following commutative diagram~: \[
\xymatrix@C=4mm{ 0 \ar[r] & \Lambda(e_{i_2}, \dotsc, e_{i_r}) \ar[r] 
\ar[d]^\simeq & \Lambda(e_{i_2}, \dotsc, e_{i_r}) \oplus \oplus_{s \geq 
1} \Q u_s \ar[r] \ar[d]^{H^\star \theta} & \oplus_{s \geq 1} \Q u_s 
\ar[d]^\simeq \ar[r] & 0 \\ 0 \ar[r] & H^\star(B, d) \ar[r] & H^\star(
A_2, d) \ar[r] & H^\star(A_2/B, 0) \ar[r] & 0}
\]
The 5-lemma proves that $H^\star \theta$ is an isomorphism, or, in other
words, that $\theta$ is a quasi-isomorphism.
\end{proof}

\begin{lemma} \label{lem:surj}
Let $m \de (\Lambda V, d) \to (A_3, d)$ and $m' \de (\Lambda W, d) \to (
A_2, d)$ be the Sullivan minimal models of $(A_3, d)$ and $(A_2,d)$, and
 $f \de (\Lambda V,d) \to (\Lambda W,d)$ a minimal model of the 
injection $(A_3,d) \to (A_2, d)$. Then $Qf \de V \to W$ is surjective.
\end{lemma}
\begin{proof}
Let $(v_1, v_2, \dotsc)$ be a basis of $V$. Since $de_{i_1} = 0$, 
rational homotopy theory shows that we can construct the map $m$ with 
the property that $m(v_1) = e_{i_1}$.  We form then the relative 
Sullivan model~: $(\Lambda V \otimes \Lambda a, d)$ with $da = v_1$. The
 map $m \otimes \id \de (\Lambda V \otimes \Lambda a, d) \to (A_3 
\otimes \Lambda a, d)$ extends the map $m$ and makes commutative the 
following diagram. \[
\xymatrix{(\Lambda V, d) \ar@{^{(}->}[r]^{i} \ar[d]_m & (\Lambda V \otimes \Lambda a, d) \ar[d]^{m \otimes \id} \\ (A_3, d) 
\ar@{^{(}->}[r]^{j\qquad} & (A_3 \otimes \Lambda a, d) = (A_2, d)}
\]
Since $m$ is a quasi-isomorphism, $m\otimes \id$ is also a quasi-
isomorphism (see~lemma 14.2 of~\cite{fe00}).  This shows that $m \otimes
 \id$ is a Sullivan model of the map $j \circ m$.

The relative Sullivan algebra $(\Lambda V \otimes \Lambda a, d)$ is a 
Sullivan algebra, and almost minimal~: to make it minimal, we only need to divide by the ideal generated by $a$ 
and $v_1$.  The projection map $p \de (\Lambda V \otimes \Lambda a, d) 
\to (\Lambda(v_2, v_3, \dotsc), d)$ is such a quasi-isomorphism.  So,  
$(\Lambda(v_2, v_3, \dotsc), d)$ is a minimal model of $(A_2, d)$.  We 
conclude by letting $f = p \circ i$. The map $f$ is such that the linear
 map $Qf$ is simply the projection $V \to V/v_1$, which is surjective. 
\end{proof}

\begin{lemma} \label{lem:retlie}
Let $(\Lambda V, d)$ be a minimal algebra and $f \de (\Lambda V, d) \to 
(E, d)$ be a quasi-isomorphism of differential graded algebras. If there
exists $x, y \in V$ such that $x$ and $y$ are linearly independant, $dx
= dy = 0$ and $f(xy) = f(x^2) = f(y^2) = 0$, then there exists two 
morphisms of Lie algebras $\Lie(u,v) \stackrel{i}{\to} L_{(\Lambda V, 
d)} \stackrel{p}{\to} \Lie(u,v)$ such that $p \circ i = \id$. In particular, $i$ is injective.
\end{lemma}
\begin{proof}
Let's consider the differential graded algebra $(B, 0) = (\Q \oplus \Q 
x' \oplus \Q y', 0)$ with all products equal to zero and $|x'| = |x|, 
|y'| = |y|$.  We can define a morphism of differential graded algebras $\theta \de (B,0) \to 
(E,d)$ with $\theta(x') = f(x)$ and $\theta(y') = f(y)$.

Notice that $(B,0)$ is a model of a wedge of two spheres. Its minimal 
Sullivan model $\rho \de (\Lambda W, d) \to (B,0)$ is such that 
$L_{(\Lambda W, d)} = \Lie(u,v)$ with $|u| = 
|x'| - 1$ and $|v| = |y'| - 1$. Without loss of generality, we 
can assume that $|x'| \leq |y'|$.

The existence of the Sullivan minimal model is proved by an inductive 
process.  Looking closely at this construction, we can easily (in low 
degree) construct a basis for $W$. 
\begin{itemize}
\item If $|x'|$ is odd or if $|x'| = |y'|$, then $\Lambda W = \Lambda
(x', y', t, \dotsc)$ with $dt = x'y'$. In degree less than $|y'|$, $W$ 
has only two generators~: $x', y'$.
\item If $|x'|$ is even and if $|x'| <|y'|$, then $\Lambda W = \Lambda
(x', y', t_1, t_2, \dotsc)$ with $dt_1 = x'^2$ and $dt_2 = x'y'$.
\end{itemize}
Let's construct a map $\psi \de (\Lambda W, d) \to (\Lambda V, d)$.  By 
the lifting lemma, such a map can be obtained by lifting $\theta \circ 
\rho$ along $f$.  But we can have more~: the lift $\psi$ can be 
constructed inductively along a basis of $W$, so we can choose $\psi(x') = x$ and $\psi(y') = y$. 
 \[
\xymatrix{ & (\Lambda V, d) \ar[d]^{f} \\ (\Lambda W, d) \ar[r]_{\theta 
\circ \rho} \ar@{-->}[ur]^\psi & (E, d)}
\]
Now, let's see what happens for the induced map 
$L_{(\Lambda V, d)} \to \Lie(u, v)$.
\begin{itemize}
\item If $|x'|$ is odd or if $|x'| = |y'|$, then the linear map $Q\psi 
\de W \to V$ is injective in degree $\leq |y'|$ (it is completely 
described by $Q\psi(x') = x$ and $Q\psi(y') = y$). So, the dual map is 
surjective.  It implies that $L\psi \de L_{(\Lambda V, d)} \to \Lie(u, v)$ is surjective in degree $\leq |v|$, which
means that $u$ and $v$ are in the image of $L\psi$.
\item If $|x'|$ is even and if $|x'| <|y'|$, then we can do exactly the 
same reasoning if $|t_1| > |y'|$.  If $|t_1| \leq |y'|$, then there is a 
slight difference. In that case, $x^2 = \psi(x'^2) = \psi(dt_1) = d\psi(t_1)
$. So, $x^2$ is a boundary. There is a $z \in V$ such that $x^2 = dz$.  
The map $Q\psi$ in degree $\leq |y'|$ is completely described by $Q\psi(
x') = x, Q\psi(y') = y$ and $Q\psi(t_1) = z$.  It is injective in degree 
$\leq |y'|$. So, the dual map is surjective in degree $\leq 
|v|$, which also means that $u$ and $v$ are in 
the image of $L\psi$.
\end{itemize}
In both cases, the map $L\psi \de L_{(\Lambda V, d)} \to \Lie(u,
 v)$ has $u$ and $v$ in its image.  Therefore, 
we can choose $a, b \in L_{(\Lambda V, d)}$ such that $L\psi(a) = 
u$ and $L\psi(b) = v$.  Let $p = L\psi$ and consider the
map $i \de \Lie(u, v) \to L_{(\Lambda V, d)}$ defined by
$i(u) = a$ and $i(v) = b$.  These two maps verify $p 
\circ i = \id$.
\end{proof}

Now, the preliminary work is done. The main lemmas of this section can 
be proved.

\begin{lemma} \label{lem:a}
There exists an injective map $\Lie(u,v) \to L_{(A_4,0)}$.
\end{lemma}
\begin{proof}
Let $L_1 = L_{(A_1,0)}$ and $L_2 = L_{(A_4,0)}$. The proof will be done 
by showing the existence of two injective maps  \[
\Lie(u,v) \stackrel{g_1}{\to} L_1 \stackrel{g_2}{\to} L_2.
\]
\emph{Step 1~: constructing the map $g_2$.} 
By lemma~\ref{lem:algrel}, $(A_1,0) \stackrel{\simeq}{\to} (A_2,d)$ and 
$(A_4,0) \stackrel{\simeq}{\to} (A_3,d)$, so $L_1 = L_{(A_2, 0)}$ and 
$L_2 = L_{(A_3,d)}$. The lemma~\ref{lem:surj} gives us a map $f \de 
(\Lambda V, d) \to (\Lambda W, d)$ between the Sullivan minimal models of $(A_3,d)$ and $(A_2,d)$. Applying the homotopy Lie algebra 
functor to the map gives a map $Lf \de L_1 \to L_2$. The surjectivity of
$Qf$ implies that $Lf$ is injective (see~\cite{fe00}, chapter 21). Now,
$g_2 = Lf$ is the required map.

\emph{Step 2~: constructing the map $g_1$.} By lemma~\ref{lem:retlie}, 
we only need to show that if $m \de (\Lambda V, d) \to (A_1, 0)$ is a 
Sullivan minimal model, then there exists $x,y \in V$ such that $x,y$ 
are linearly independant, $dx = dy = 0$ and $m(xy) = m(x^2) = m(y^2) 
= 0$.

Since $m$ is a quasi-isomorphism, $H^\star m \de H^\star(\Lambda V, d) 
\to (A_1, 0)$ is surjective. So, there exists $[x]$ and $[y]$ in 
$H^\star(\Lambda V, d)$ such that $H^\star m([x]) = e_{i_2}$ and 
$H^\star m([y]) = u_1$. It gives us $x$ and $y$ in $(\Lambda V, d)$ such
that $dx = dy = 0$, $m(x) = e_{i_2}$ and $m(y) = u_1$.  But $x$ and $y$ 
can not be in $\Lambda^{\geq 2} V$ because, otherwise, $e_{i_2} = m(x)$ 
would be in $\Lambda^{\geq 2}(e_{i_2}, \dotsc, e_{i_r})$ and $u_1 = 
m(y)$ would be in $\Lambda^{\geq 2}(u_1)/u_1^2$. Therefore, $x$ and $y$ 
are in $V$.  Finally, the lemma~\ref{lem:retlie} gives us the map $g_1$.
\end{proof}

\begin{lemma} \label{lem:b}
There exists an injective map $\Lie(u,v) \to L_{(A_5,0)}$.
\end{lemma}
\begin{proof}
Recall that $A_5$ is the quotient of $\Lambda(e_{j_1}, \dotsc, e_{j_m})$
 by the ideal $I$ generated by the elements $e_{i_1} \dotso e_{i_{r+1}}$
 and $[e_{i_1}, \dotsc, e_{i_{r+1}}]$.  It is clear that $A_5^{>r} = 0$.
Let's prove that a basis of $A_5^r$ is given by the classes of the 
elements $e_1e_{j_2} \dotso e_{j_r}$ with $1 < j_2 < \dotsc < j_r \leq 
m$.
\begin{itemize}
\item Let $e_{i_1}\dotso e_{i_r} \in A_5^r$, with $i_1 < \dotso < i_r$. 
If $i_1 = 1$, then it is trivially a linear combination of elements $e_1 e_{j_2} 
\dotso e_{j_r}$. If $i_1 > 1$, then we know that $[e_1, e_{i_1}, \dotsc,
e_{i_r}] = 0$. So, it is also a linear combination of such elements.  It
 shows that these elements generate $A_5^r$.
\item If $1 < i_1 < \dotso < i_{r+1} \leq m$, then
\begin{align*}
\sum_{j=1}^{r+1} (-1)^{j+1}[e_1, e_{i_1}, & \dotsc, \widehat{e_{i_j}}, 
\dotsc, e_{i_{r+1}}] = \sum_{j=1}^{r+1} (-1)^{j+1}\bigg( -e_{i_1}\dotso 
\widehat{e_{i_j}} \dotso e_{i_{r+1}} \\
&\quad + \sum_{k=1}^{j-1} (-1)^{k+1} e_1e_{i_1} \dotso \widehat{e_{i_k}}
\dotso \widehat{e_{i_j}} \dotso e_{i_{r+1}} \\
& \quad + \sum_{k=j+1}^{r+1} (-1)^{k} e_1 e_{i_1} \dotso \widehat{
e_{i_j}} \dotso \widehat{e_{i_k}} \dotso e_{i_{r+1}}\bigg) \\
& = \sum_{j=1}^{r+1} (-1)^{j}e_{i_1}\dotso \widehat{e_{i_j}} \dotso 
e_{i_{r+1}} =[e_{i_1}, \dotsc, e_{i_{r+1}}].
\end{align*}
It shows that the vector space generated by every $[e_{i_1}, \dotsc, 
e_{i_{r+1}}]$ is equal to the vector space generated by the elements 
$[e_1, e_{i_2}, \dotsc, e_{i_{r+1}}]$ with $1 < i_2 < \dotso < i_{r+1}$.
\item Let's consider the following short exact sequence \[
0 \to \langle [e_{i_1}, \dotsc, e_{i_{r+1}}] \rangle \to \Lambda^r(
e_{j_1}, \dotsc, e_{j_m}) \to A_5^r \to 0.
\]
Let $d_1$ be the dimension of the vector space generated by the elements
 $e_1 e_{i_2} \dotso e_{i_r}$, with $1 < i_2 < \dotso < i_r$, in $\Lambda^r(e_{j_1}, 
\dotsc, e_{j_m})$ and $d_2$ be the dimension of the vector space 
generated by the elements $e_{i_1} e_{i_2} \dotso e_{i_{r}}$ with $1 <
 i_1 < \dotso <i_r$.  We have~: $\dim \Lambda^r(e_{j_1}, \dotsc, e_{j_m}) = d_1 + d_2$
, $\dim A_5^r \leq d_1$, $\dim \langle [e_{i_1}, \dotsc, e_{i_{r+1}}] 
\rangle \leq d_2$. So, $\dim A_5^r = d_1$, and the elements $e_1 e_{j_2}
\dotso e_{j_{r}}$ form a basis of $A_5^r$.
\end{itemize}
Let $I$ be the set of every sequence $1 < i_1 < \dotso < i_{r+1}$ and $(B,d)$ the differential graded algebra defined by \[
B = \frac{\Lambda(e_{j_1}, \dotsc, e_{j_m})}{\Lambda^{>r}(e_{j_1}, 
\dotsc, e_{j_m})} \oplus \left(\oplus_{i \in I} a_i \right)
\]
and $d(a_i) = [e_1, e_{i_2}, \dotsc, e_{i_{r+1}}]$.  The product in $B$ 
is defined by $a_i \cdot a_j = a_i \cdot e_j = 0$. The ideal generated
by the $a_i$ et the $da_i$ is acyclic, and the quotient map is a quasi-
isomorphism~: $\phi \de (B,d) \to (A_5, 0)$.

Therefore, the differential graded algebras $(A_5,0)$ and $(B,d)$ have 
the same minimal model.  Let's construct the minimal model of $(B,d)$. 
We want $\theta \de (\Lambda W, d) \to (B,d)$.   $W = (e_1, \dotsc, e_m,
a_I, \dotsc)$ with $\theta(e_i) = e_i$, $\theta(a_i) = a_i$.  Because
$\theta(e_i^2) = \theta(e_i a_j) = \theta(a_j^2) = 0$, 
lemma~\ref{lem:retlie} shows that $L_{(\Lambda W, d)} = L_{(A_5,0)}$ 
contains a Lie subalgebra $\Lie(u,v)$.
\end{proof}

\begin{lemma} \label{lem:magic}
If $(A,0)$ is a 1-connected differential graded algebra, $X$ is a formal
 space and if there exists two maps $f\de A \to H^\star(X, \Q)$ and $g
\de H^\star(X, \Q) \to A$ such that $g \circ f = \id_A$, then there 
exists two morphisms of Lie algebras $\tilde{f} \de L_X \to L_{(A,0)}$ 
and $\tilde{g} \de L_{(A,0)} \to L_X$ such that $\tilde{f} \circ 
\tilde{g}$ is an isomorphism. In particular, $\tilde{g}$ is an injective
map.
\end{lemma}
\begin{proof}
Let $(\Lambda V, d) \stackrel{m}{\to} \left(A, 0\right)$ and $(\Lambda 
V', d') \stackrel{m'}{\to} \left(H^\star(X, \Q),0\right)$ be the minimal
Sullivan models of $(A,0)$ and $X$ respectively (the map $m'$ exists 
because $X$ is a formal space).  Since these maps are quasi-isomorphisms
, they are surjective. The lifting lemma shows that there exists maps 
$\bar{f}$ and $\bar{g}$ such that $m' \circ \bar{f} = f \circ m$ and $m 
\circ \bar{g} = g \circ m'$. \[
\xymatrix{(\Lambda V, d) \ar[r]^{\bar{f}} \ar[d]_m & (\Lambda V', d') 
\ar[r]^{\bar{g}} \ar[d]_{m'} & (\Lambda V, d) \ar[d]^m \\
(A,0) \ar[r]_f & (H^\star(X, Q), 0) \ar[r]_g & (A,0) }
\]

The maps $\bar{f}$ and $\bar{g}$ verify $m \circ (\bar{g} \circ \bar{f})
 = (g \circ f) \circ m = m$. Since $g \circ f$ is an isomorphism, 
$\bar{g} \circ \bar{f}$ is a quasi-isomorphism between 1-connected 
minimal Sullivan algebras. It implies that it is an isomorphism.

Applying the homotopy Lie algebra functor to $(\Lambda V, d) \stackrel{
\bar{f}}{\to} (\Lambda V', d') \stackrel{\bar{g}}{\to} (\Lambda V, d)$ 
gives us the maps $\tilde{f} = L\bar{f}$ and $\tilde{g} = L\bar{g}$. \[
L_{(A,0)} \stackrel{\tilde{f}}{\leftarrow} L_X \stackrel{\tilde{g}}{
\leftarrow} L_{(A,0)} 
\]
Since $L$ is a functor, $\tilde{f} \circ \tilde{g} = (L\bar{f}) \circ (L
\bar{g}) = L(\bar{g} \circ \bar{f})$ is an isomorphism.
\end{proof}

\end{document}